\documentclass[12pt,leqno,amscd, amsfonts, amssymb,pstricks,verbatim]{amsart}
\usepackage{amsfonts}
\usepackage{amscd}
\usepackage{mathrsfs}
\usepackage{amsmath}
\usepackage{amssymb}
\usepackage{hyperref}
\usepackage{bookmark}

\def\a{\alpha}

\def\Z{\mathbb{Z}}
\def\N{\mathbb{N}}

\def\C{\mathbb{C}}

\numberwithin{equation}{section}
\newtheorem{theo}{Theorem}[section]
\newtheorem{defi}[theo]{Definition}
\newtheorem{coro}[theo]{Corollary}
\newtheorem{lemm}[theo]{Lemma}

\newtheorem{clai}{Claim}
\newtheorem{rema}[theo]{Remark}

\newtheorem{exam}{Example}
\allowdisplaybreaks

\begin{document}

\title[Tensor product weight modules over the affine-Virasoro algebra]{Tensor product weight modules over the affine-Virasoro algebra}

\author{Qiu-Fan Chen}

\address{Chen: Department of Mathematics, Shanghai Maritime University,
 Shanghai, 201306, China.}\email{chenqf@shmtu.edu.cn}

\author{Yu-Feng Yao}

\address{Yao: Department of Mathematics, Shanghai Maritime University,
 Shanghai, 201306, China.}\email{yfyao@shmtu.edu.cn}



\thanks{This work is supported by National Natural Science Foundation of China (Grant Nos. 12271345  and 12071136).}

\begin{abstract}
In this paper, we study the tensor products of irreducible highest weight modules with irreducible loop modules over the affine-Virasoro algebra with aid of the ``shifting technique" established for the Virasoro algebra in \cite{CGZ}.
All such tensor product modules are indecomposable modules with infinite-dimensional weight spaces. Moreover, we obtain the necessary and sufficient conditions for such tensor product modules to be irreducible. Therefore, we obtain a class of new irreducible weight modules over  the affine-Virasoro algebra. Finally, the necessary and sufficient conditions for any two such tensor product modules to be isomorphic are also determined.
\end{abstract}

\maketitle

\setcounter{tocdepth}{1}\tableofcontents
\begin{center}
\end{center}

\section{Introduction}
Throughout the paper, we denote by $\C ,\,\Z,\,\C^*,\,\Z_+,\,\N$ the sets of complex numbers, integers, nonzero complex numbers, nonnegative integers and positive integers, respectively. All algebras (modules, vector spaces) are assumed to be  over $\C$. For a Lie algebra $\mathfrak{g}$, we use $U(\mathfrak{g})$ to denote the universal enveloping algebra of $\mathfrak{g}$. Let $\C[t^{\pm1}]$ and  $\C[t]$ be the algebras of Laurent polynomials and of polynomials in one indeterminate $t$, respectively.

Let $\mathfrak{g}$ be a finite-dimensional simple Lie algebra equipped with the Killing form $(\cdot\mid\cdot)$. Associated to the pair $(\mathfrak{g},(\cdot\mid\cdot))$, we have the corresponding untwisted affine Kac-Moody Lie algebras $\tilde{\mathfrak{g}}$. Affine Lie algebras are the most useful ones among infinite-dimensional Kac-Moody Lie algebras, so that they are extensively studied, such as representation theory \cite{K1}, vertex operator algebra theory \cite{LL} and conformal field theory \cite{FMS}. It is worth noting that their representation theory is as rich as but quite different to that of finite-dimensional simple Lie algebras. Denote by $\mathcal{V}$ the Virasoro algebra, which is a central extension of the Lie algebra of complex vector fields on the circle. The algebra $\mathcal{V}$ is one of the most important Lie algebras both in mathematics and in mathematical physics, see for example \cite{IK,KR} and references therein. The Virasoro algebra acts on any (except when the level is negative the dual coxeter number) highest weight module of the affine Lie algebra through the use of the famous Sugawara operators.  Conversely, the affine Lie algebras admit representations on the Fock space and hence admit representations of the Virasoro algebra. This close
relationship strongly suggests that they should be considered simultaneously, i.e., as one algebraic structure, and hence has led to the definition of the so-called affine-Virasoro algebra $\mathfrak{L(g)}$. It is known that in literature this algebra is also named the conformal current algebra \cite{CK,K}, the entire gauge algebra \cite{EFK}. The physical context in which the affine-Virasoro algebra appears
is a two-dimensional conformal field theory on the circle with an internal symmetry algebra. In particular, the even part of  the $N=3$ superconformal algebra \cite{CL} is just $\mathfrak{L(sl_2)}$.
We know that physical applications of algebraic structure usually appear through representations. In recent years, the representation theory of affine-Virasoro algebras generally captures the attention of people. Highest weight representations, integrable  representations and $U(\mathfrak{h})$-free representations of the affine-Virasoro algebras have been  extensively studied  (cf.~\cite{B,CH,EJ,JY,LQ,XH}. Especially, in \cite{LPX}, the authors  gave a complete classification of irreducible quasi-finite $\mathfrak{L(g)}$-modules. It is shown that they are all highest weight modules, lowest weight modules and the loop modules.
Recently, the author classified a class of irreducible $\mathfrak{L(g)}$-modules on which each weight vector of the positive parts acts locally finitely in \cite{T}.

It is well-known that an important way to construct new modules over an algebra is to consider the linear tensor product of known modules
over the algebra (cf.~ \cite{CGZ,CHSY,CY1,CY2,GLW,TZ1,TZ2,Z}). In the paper \cite{Z}, Zhang considered the tensor products of irreducible highest weight modules with irreducible intermediate series modules, and he provided some sufficient conditions for the tensor products to be irreducible. It was even not known in \cite{Z} whether there was a non-simple tensor product when the highest weight module is not a Verma module until \cite{CGZ}. By using the so-called ``shifting technique" and Feigin-Fuchs' theorem, the authors \cite{CGZ} gave the necessary and sufficient conditions for such tensor products to be irreducible. The idea was exploited and generalized to consider such tensor products over infinite-dimensional Lie algebras, such as the twisted Heisenberg-Virasoro algebra \cite{LZ}, the mirror Heisenberg-Virasoro algebra \cite{GZ1}, the Neveu-Schwarz algebra \cite{ZX}, the Schr\"{o}dinger-Virasoro algebra \cite{LZX} and the untwisted affine Kac-Moody algebra \cite{GZ}. The aim of this paper is to study the tensor product modules of irreducible highest weight modules with irreducible loop modules over the affine-Virasoro algebra $\mathfrak{L(g)}$.

This paper is organized as follows. In Section 2, we recall notations related to the affine-Virasoro algebra $\mathfrak{L(g)}$ and collect some known results on highest weight modules and loop modules over $\mathfrak{L(g)}$. We also introduce the ``shifting technique" for later use.  In Section 3, we prove that the tensor product modules of irreducible highest weight modules with irreducible loop modules over $\mathfrak{L(g)}$ are indecomposable.  Section 4 is devoted to determining the necessary and sufficient conditions for such tensor product modules to be irreducible. In order to illustrate our main theorem, some examples are given. Consequently, we obtain a class of new irreducible weight $\mathfrak{L(g)}$-modules with infinite-dimensional weight spaces. In Section 5, we show that two such tensor product modules are isomorphic if and only if the corresponding highest weight modules and loop modules are isomorphic, respectively.
\section{Notations and preliminaries}\label{pre}
In this section we recall some notations and collect some known results related to the affine-Virasoro algebra.

Let $\mathfrak{g}$ be a complex simple finite-dimensional Lie algebra with a Cartan subalgebra $\mathfrak{h}$. Then $\mathfrak{g}$ has a root space decomposition $\mathfrak{g}=\mathfrak{h}\oplus(\oplus_{\a\in\Delta}\mathfrak{g}_{\a})$, where $\Delta$ is the root system determined by  $\mathfrak{h}$, and $\mathfrak{g}_{\a}$ is the root space corresponding to the root $\a\in\Delta$ with ${\rm dim\,}\mathfrak{g}_{\a}=1$. Denote by $\Delta_{+}$ and $\Delta_{-}$  the set of positive roots and negative roots, respectively. Then $\Delta_{-}=-\Delta_{+}$ and $\Delta=\Delta_{+}\cup \Delta_{-}$. Let $\mathfrak{n}_{\pm}=\oplus_{\a\in\Delta_{\pm}}\mathfrak{g}_{\a}$. Clearly, $\mathfrak{g}$ admits a triangular decomposition $\mathfrak{g}=\mathfrak{n}_{-}\oplus\mathfrak{h}\oplus\mathfrak{n}_{+}$ in the sense of \cite{MP}. Let
 $\Lambda^{+}$ denote the set of the integral dominant weights of  $\mathfrak{g}$, and $(\cdot\mid\cdot)$  the Killing form on $\mathfrak{g}$ normalized by the condition $(\theta\mid \theta)=2$, where $\theta\in \Delta_{+}$ is the highest root. Let $\rho$ be the half-sum of positive roots of $\mathfrak{g}$. For $\lambda\in\mathfrak{h}^*$, denote $c_{\lambda}=(\lambda+2\rho\mid \lambda)$ for simplicity, which is the eigenvalue of the Casimir operator acts on the simple $\mathfrak{g}$-module with highest weight $\lambda$. Especially, the number $g=\frac{1}{2}c_{\theta}$ is called the dual Coxeter number and it is a positive integer.

The {\em  untwisted affine Kac-Moody Lie algebras $\tilde{\mathfrak{g}}$} is a central extension of the loop algebra $\mathfrak{g}\otimes\C[t^{\pm1}]$
defined by the following  Lie brackets:
$$[x\otimes t^m, y\otimes t^n]=[x,y]\otimes t^{m+n}+m(x\mid y)\delta_{m+n,0}\mathbf{k},\quad [\mathbf{k}, \tilde{\mathfrak{g}}]=0,\ \ x,y\in\mathfrak{g}, m,n\in\Z.$$
For convenience, we will denote  $x\otimes t^m$ by $x(m)$ for any $x\in\mathfrak{g}$ and $m\in\Z$. Particularly, for $m=0$, we identify $x(0)$ with $x$. Denote by $\tilde{\mathfrak{n}}_{\pm}=\mathfrak{n}_{\pm}\oplus (\mathfrak{g}\otimes t^{\pm1}\C[t^{\pm1}])$. Then $\tilde{\mathfrak{g}}$ has a structure of triangular decomposition $\tilde{\mathfrak{g}}=\tilde{\mathfrak{n}}_{-}\oplus\tilde{\mathfrak{h}}\oplus\tilde{\mathfrak{n}}_{+}$, where $\tilde{\mathfrak{h}}=\mathfrak{h}+\C\mathbf{k}$.

The {\em Virasoro algebra $\mathcal{V}$} is a central extension of the Lie algebra ${\rm Der\,}\C[t^{\pm1}]$ with the following  Lie brackets:
$$[d_m, d_n]=(n-m)d_{m+n}+\delta_{m+n,0}\frac{m^3-m}{12}\mathbf{c},\quad [\mathbf{c}, \mathcal{V}]=0,\ \ m,n\in\Z,$$
where $d_m=t^{m+1}\frac{d}{dt}$. Let $\mathcal{V}_{\pm}=\oplus_{\pm n\in\N}\C d_n$ and $\mathcal{V}_0=\C d_0+\C \mathbf{c}$ (a Cartan subalgebra of $\mathcal{V}$), then $\mathcal{V}$ carries a structure of triangular decomposition $\mathcal{V}=\mathcal{V}_{-}\oplus\mathcal{V}_{0}\oplus\mathcal{V}_{+}$.
\begin{defi}(cf. \cite{K}) \label{defi2.1}
\rm
Let $\mathfrak{g}$ be a finite-dimensional simple Lie algebra. The affine-Virasoro Lie algebra corresponding to  $\mathfrak{g}$ is the algebra
$$\mathfrak{L(g)}=\tilde{\mathfrak{g}}\rtimes \mathcal{V}=(\oplus_{n\in\Z}\mathfrak{g}(n))\oplus (\oplus_{n\in\Z}\C d_n)\oplus \C\mathbf{k}\oplus\C\mathbf{c}$$ and subject to the following Lie brackets:
\begin{equation*}
\aligned
&[x(m), y(n)]=[x,y](m+n)+m(x\mid y)\delta_{m+n,0}\mathbf{k},\\
&[d_m, d_n]=(n-m)d_{m+n}+\delta_{m+n,0}\frac{m^3-m}{12}\mathbf{c},\\
&[d_n,x(m)]=mx(m+n),\\
&[\mathfrak{L(g)},\C\mathbf{k}+\C\mathbf{c}]=0,\quad \forall m,n\in\Z, x,y\in\mathfrak{g}.
\endaligned
\end{equation*}
\end{defi}
By definition, it is easy to see that $\mathcal{H}:=\mathcal{V}_0+\tilde{\mathfrak{h}}$ is the Cartan subalgebra of $\mathfrak{L(g)}$.
Then $\mathfrak{L(g)}$ admits a triangular decomposition $$\mathfrak{L(g)}=\mathfrak{L(g)}_{-}\oplus\mathcal{H}\oplus\mathfrak{L(g)}_{+},$$ where $\mathfrak{L(g)}_{\pm}=\mathcal{V}_{\pm}+\tilde{\mathfrak{n}}_{\pm}$.
In addition, $\mathfrak{L(g)}$ has a $\Z$-gradation with respect to the adjoint action of $d_0$:
$$\mathfrak{L(g)}=\oplus_{n\in\Z}\mathfrak{L(g)}_{n},\quad [\mathfrak{L(g)}_{n}, \mathfrak{L(g)}_{m}]\subset\mathfrak{L(g)}_{n+m},\,\,m,n\in\Z,$$
where $\mathfrak{L(g)}_{0}=\mathfrak{g}+\C d_0+\C\mathbf{k}+\C\mathbf{c}$ and $\mathfrak{L(g)}_{n}=\mathfrak{g}(n)+\C d_n$ for $n\neq0$. A nonzero element $x$ in $U(\mathfrak{L(g)})$ is called homogeneous if $[d_0, x]=mx$ for some $m\in\Z$, and $m$ is called the {\em degree} of $x$, denoted by ${\rm deg \,}(x)=m$. Let $U(\mathfrak{L(g)})_m$ be the subspace generated by all elements of degree $m$.
\begin{defi} \label{defi2.3}\rm Keep notations as before.
\begin{itemize}
\item[(1)] A module $V$ over $\mathfrak{L(g)}$ is called a {\em weight module} with respect to $\mathcal{H}$ if $V=\oplus_{\lambda\in\mathcal{H}^{*}}V_{\lambda}$, where
$$V_{\lambda}=\{v\in V\mid hv=\lambda(h)v,\,\,\forall h\in \mathcal{H}\}.$$
\item[(2)] If a weight $\mathfrak{L(g)}$-module $V$ is generated by a one-dimensional $\mathcal{H}\oplus\mathfrak{L(g)}_{+}$-module on which $\mathfrak{L(g)}_{+}$ acts trivially, then $V$ is called a {\em  highest weight module}.
\end{itemize}
\end{defi}

Given $\lambda\in\mathfrak{h}^*$  and $l, k, c\in\C$, we define a one-dimensional $\mathcal{H}\oplus\mathfrak{L(g)}_{+}$-module $\C u$ by
\begin{eqnarray*}
&h u=\lambda(h)u,\ \  \forall h\in\mathfrak{h},\\
&d_0 u=lu,\ \ \ \ \mathbf{k}u=ku,\ \ \ \ \mathbf{c} u=cu,\\
&\mathfrak{g}_{\a} u=0, \ \  \forall \a\in\Delta_+,\\
&x(n)u=d_nu=0,\ \ \forall x\in\mathfrak{g}, n\in\N.
\end{eqnarray*}
The Verma module over $\mathfrak{L(g)}$ is defined by
$$M(\lambda, l, k, c)=U(\mathfrak{L(g)})\otimes_{U(\mathcal{H}\oplus\mathfrak{L(g)}_{+})}\C u.$$
One knows that $M(\lambda, l, k, c)$ has a unique maximal submodule $J(\lambda, l, k, c)$, and the corresponding irreducible quotient module is denoted by $L(\lambda, l, k, c)$. We will use $u, \bar{u}$ to denote the highest weight vector of  $M(\lambda, l, k, c)$ and $L(\lambda, l, k, c)$, respectively. A nonzero weight vector $v\in M(\lambda, l, k, c)$ is called a singular vector if $\mathfrak{L(g)}_{+}v=0$.

Similarly, we shall denote by $M_{\mathcal{V}}(l, c)$ the Verma module over the
Virasoro algebra $\mathcal{V}$, by $M_{\tilde{\mathfrak{g}}}(\lambda, k)$ the Verma module over the untwisted affine Kac-Moody Lie algebras $\tilde{\mathfrak{g}}$. Let $u_{l,c}, u_{\lambda,k}$ be the respective highest weight vectors of  $M_{\mathcal{V}}(l, c), M_{\tilde{\mathfrak{g}}}(\lambda, k)$.  Let $L_{\tilde{\mathfrak{g}}}(\lambda, k), L_{\mathcal{V}}(l, c)$ denote the corresponding irreducible quotient modules of  $M_{\tilde{\mathfrak{g}}}(\lambda, k)$ and  $M_{\mathcal{V}}(l, c)$, respectively. A $\mathcal{V}$-module $M_{\mathcal{V}}(l, c)$ can be viewed as an $\mathfrak{L(g)}$-module by defining trivial  $\tilde{\mathfrak{g}}$-action and we denote by $M_{\mathcal{V}}^{\mathfrak{L(g)}}(l, c)$ the resulting $\mathfrak{L(g)}$-module. Take a basis $\{x_i\mid 1\leq i\leq \dim\mathfrak{g}\}$ and its dual basis $\{y_i\mid 1\leq i\leq \dim\mathfrak{g}\}$ of $\mathfrak{g}$ with respect to the Killing form $(\cdot\mid\cdot)$, i.e., $(x_i\mid y_j)=\delta_{i,j}$ for any $i, j$. For any $k\in\C, \lambda\in\mathfrak{h}^*$,  we define the linear
operators
$$T_{n}=-\frac{1}{2}\sum_{j\in\Z}\sum_{i}:x_i(-j)y_i(j+n):$$
on the $\tilde{\mathfrak{g}}$-module $M_{\tilde{\mathfrak{g}}}(\lambda, k)$,
where
$$:x_i(-j)y_i(j+n):\stackrel{\bigtriangleup}{=}
\begin{cases}
x_i(-j)y_i(j+n), &\text{if}\,\,-j\leq j+n;\cr
y_i(j+n)x_i(-j), &\text{if}\,\,-j>j+n.
\end{cases}$$
It follows from \cite{K} that
\begin{align*}
&[T_m, T_n]=(n-m)(k+g)T_{m+n}+\delta_{m+n,0}\frac{m^3-m}{12}({\rm dim \,}\mathfrak{g})(k+g)k,\\
&[T_n, x(m)]=(k+g)mx(m+n),\ \ \ \ m,n\in\Z.
\end{align*}
For the case $k\neq -g$, $M_{\tilde{\mathfrak{g}}}(\lambda, k)$ becomes an $\mathfrak{L(g)}$-module by the following Sugawara operators
\begin{align*}
&d_m\mapsto \frac{1}{k+g}T_m, \ \ \ \forall m\in\Z,\\
&x(n)\mapsto x(n), \ \ \ \forall n\in\Z,\\
&\mathbf{c}\mapsto \frac{k({\rm dim \,}\mathfrak{g})}{k+g},\\
&\mathbf{k}\mapsto k.
\end{align*}
We denote by $M_{\tilde{\mathfrak{g}}}^{\mathfrak{L(g)}}(\lambda, k)$ the resulting $\mathfrak{L(g)}$-module.

%\subsection{Loop module}
Let $L(\mu)$ be the finite dimensional irreducible $\mathfrak{g}$-module with the highest weight $\mu\in\Lambda^{+}$ and  the associated highest weight vector $v_{\mu}$. Now we recall from \cite{EJ} the  irreducible loop modules $L_{a,b}(\mu)=L(\mu)\otimes \C[t^{\pm1}]$ over $\mathfrak{L(g)}$ for $a,b\in\C$. The action of  $\mathfrak{L(g)}$ on the module $L_{a,b}(\mu)$ is as follows
\begin{eqnarray*}
&d_m(v\otimes t^n)=(a+bm+n)v\otimes t^{m+n},\\
&x(m)(v\otimes t^n)=(xv)\otimes t^{m+n},\\
&\mathbf{k}(v\otimes t^n)= \mathbf{c}(v\otimes t^n)=0,\ \ \ \ \forall m,n\in\Z, x\in\mathfrak{g}, v\in L(\mu).
\end{eqnarray*}
Notice that $L_{a,b}(\mu)\cong L_{a+n,b}(\mu)$ for all $n\in\Z$. To avoid repetition, throughout the paper, we always assume that $0\leq{\rm Re\,}a<1$, where ${\rm Re\,}a$ is the real part of $a$. It has been shown in \cite{LPX} that  any irreducible  $\mathfrak{L(g)}$-module with finite-dimensional weight spaces with respect to $d_0$ is  a highest weight module, a lowest weight module, or isomorphic to some $L_{a,b}(\mu)$.

%Suppose that $V$ is a finite-dimensional $\mathfrak{g}$-module. It is well known that $T_{a,b}(V)$ is irreducible if and only if $V$ is a non-trivial irreducible module or $a\not\in\Z$ or $b\not\in\{0,1\}$ \cite{EJ}. We assume that $V=L(\mu)$ is an irreducible $\mathfrak{g}$-module with the highest weight $\mu\in\Lambda^{+}$ and  the associated highest weight vector $v_{\mu}$.  Denote by $L_{a,b}(\mu)$  the irreducible submodule or subquotient of $T_{a,b}(L(\mu))$.

Form the tensor product $\mathfrak{L(g)}$-module $L(\lambda, l, k, c)\otimes L_{a,b}(\mu)$, which is a weight module and all weight spaces are infinite-dimensional,
\begin{equation}\label{nnmz}\big(L(\lambda, l, k, c)\otimes L_{a,b}(\mu)\big)_{a+l+n}=\sum_{i\in\Z_+}\big(L(\lambda, l, k, c)_{l-i}\otimes (L(\mu)\otimes t^{n+i})\big).
\end{equation}
Now we use the ``shifting technique". Define a new action of  $\mathfrak{L(g)}$ on the vector space $L_{\lambda, l, k, c}^{\mu, a,b}=L(\lambda, l, k, c)\otimes L(\mu)\otimes \C[t^{\pm1}]$ by
\begin{eqnarray}\label{equ1}
&d_m(P\bar{u}\otimes v\otimes t^n)=(d_m+a+n-{\rm deg\,}(P)+bm)P\bar{u}\otimes v\otimes t^{m+n},\nonumber\\
&x(m)(P\bar{u}\otimes v\otimes t^n)=(x(m)P\bar{u}\otimes v+P\bar{u}\otimes xv)\otimes t^{m+n},\nonumber\\
&\mathbf{k}(P\bar{u}\otimes v\otimes t^n)=k(P\bar{u}\otimes v\otimes t^n),\ \ \ \  \mathbf{c}(P\bar{u}\otimes v\otimes t^n)=c(P\bar{u}\otimes v\otimes t^n),
\end{eqnarray}
where $v\in L(\mu), x\in \mathfrak{g},m, n\in\Z$ and homogeneous element $P\in U(\mathfrak{L(g)}_{-})$ .
%Taking $m=0$ in the first formula gives
%$$d_0(Pu\otimes v\otimes t^n)=(a+l+n)(Pu\otimes v\otimes t^n).$$
It is not hard to check that  $L_{\lambda, l, k, c}^{\mu, a,b}$ is an $\mathfrak{L(g)}$-module and $L(\lambda, l, k, c)\otimes L_{a,b}(\mu)\cong L_{\lambda, l, k, c}^{\mu, a,b}$ via the following map:
\begin{equation}\label{equ21}L(\lambda, l, k, c)\otimes L_{a,b}(\mu)\longrightarrow L_{\lambda, l, k, c}^{\mu, a,b},\ \ \ \  P\bar{u}\otimes v\otimes t^n\longmapsto P\bar{u}\otimes v\otimes t^{n+{\rm deg\,}(P)}\end{equation}
for any $n\in\Z$ and homogeneous element $P\in  U(\mathfrak{L(g)}_{-})$.

Henceforth we will consider the module $L_{\lambda, l, k, c}^{\mu, a,b}$ instead of $L(\lambda, l, k, c)\otimes L_{a,b}(\mu)$. Clearly,  $L_{\lambda, l, k, c}^{\mu, a,b}$ is a weight module. The key advantage of the  ``shifting technique" of the notation $L_{\lambda, l, k, c}^{\mu, a,b}=\oplus_{n\in\Z}L(\lambda, l, k, c)\otimes L(\mu)\otimes t^n$ is the weight space decomposition with respect to $d_0$, i.e.,
$$L(\lambda, l, k, c)\otimes L(\mu)\otimes t^n=\{x\in L_{\lambda, l, k, c}^{\mu, a,b}\mid d_0x=(a+l+n)x\}, \ \ \ \forall n\in\Z,$$
which differs from \eqref{nnmz}.
\section{Indecomposable weight modules}
\begin{lemm}\label{lemma1}
Keep notations as before, then as an $\mathfrak{L(g)}$-module, $L_{\lambda, l, k, c}^{\mu, a,b}$ is generated by $\{\bar{u}\otimes v_{\mu}\otimes t^m\mid m\in\Z\}$.
\end{lemm}
\begin{proof}
One can observe that $L_{\lambda, l, k, c}^{\mu, a,b}$ is spanned by $\{x \bar{u}\otimes y v_{\mu}\otimes t^m\mid x\in U(\mathfrak{L(g)}_{-}), y\in U(\mathfrak{g}_{-}), m\in\Z\}$. Let $M$ be the $\mathfrak{L(g)}$-module generated by  $\{\bar{u}\otimes v_{\mu}\otimes t^m\mid m\in\Z\}$. We need to show that $M=L_{\lambda, l, k, c}^{\mu, a,b}$. It suffices to prove that
$$y\bar{u}\otimes L(\mu)\otimes t^m\in M,\,\,\forall\,y\in U(\mathfrak{L(g)}_{-}), m\in\mathbb{Z}.$$

Using \eqref{equ1}, we have
$$\mathfrak{g}_{\a}(n)(\bar{u}\otimes v_{\mu}\otimes t^m)=\bar{u}\otimes \mathfrak{g}_{\a}v_{\mu}\otimes t^{m+n},\ \ \ \ \forall m\in\Z, n\in\N,\a\in\Delta_-,$$
from which one can inductively show that $\{u\otimes L(\mu)\otimes t^m\mid m\in\Z \}\subset M$. This along with
$$d_{-n}(\bar{u}\otimes v\otimes t^m)=(d_{-n}+a+m-bn)\bar{u}\otimes v\otimes t^{m-n},\ \ \ \ \forall m\in\Z, n\in\N,v\in L(\mu)$$
gives rise to $\{d_{-n}\bar{u}\otimes L(\mu)\otimes t^{m}\mid m\in\Z, n\in\N\}\subset M$.
Inductively, we get
\begin{equation}\label{equ2}
\{U(\mathcal{V}_{-})\bar{u}\otimes L(\mu)\otimes t^{m}\mid m\in\Z\}\subset M.
\end{equation}
For any $m\in\Z$, $y\in U(\mathfrak{L(g)}_{-}),v\in L(\mu)$, we have
$$x(-n)(y\bar{u}\otimes v\otimes t^m)=(x(-n)y\bar{u}\otimes v+y\bar{u}\otimes xv)\otimes t^{m-n}$$
for any $n\in\N, x\in\mathfrak{g}$ or $n=0,x\in\mathfrak{n}_-$.
Combining these with \eqref{equ2}, one can show the lemma holds  by induction.
\end{proof}
The following theorem shows that any $L_{\lambda, l, k, c}^{\mu, a,b}$ is an indecomposable $\mathfrak{L(g)}$-module provided that $\mu\neq 0$.
\begin{theo}\label{prop1}
Keep notations as before, and in addition we assume that $\mu\neq 0$, then
${\rm End\,}_{\mathfrak{L(g)}}(L_{\lambda, l, k, c}^{\mu, a,b})\cong\C.$
\end{theo}
\begin{proof}
Let $\eta\in {\rm End\,}_{\mathfrak{L(g)}}(L_{\lambda, l, k, c}^{\mu, a,b})$. Take any but fixed $m\in\Z$. Note that $\eta(\bar{u}\otimes  v_{\mu}\otimes t^m)$ and $\bar{u}\otimes v_{\mu}\otimes t^m$ have the same weight with respect to $d_0$, we can write
$$\eta(\bar{u}\otimes v_{\mu}\otimes t^m)=\sum_{i=1}^r P_i\bar{u}\otimes v_i\otimes t^m$$
with $0\neq v_i\in L(\mu)$ for $i=1,\cdots, r$, and homogeneous elements $P_1,\ldots, P_r\in U(\mathfrak{L(g)}_{-})$ such that $0\geq {\rm deg\,}P_1\geq \cdots \geq {\rm deg\,}P_r$ and $P_1u,\ldots, P_ru$ are linearly independent. The nontriviality of $L(\mu)$ ensures that $xv_1\neq0$ for some $x\in\mathfrak{g}$. Let us take $n\in\N$ sufficiently large. It follows from direct calculation that $((a+m+n+bn)y(2n)-d_ny(n))(\bar{u}\otimes v_{\mu}\otimes t^m)=0$ for any $y\in\mathfrak{g}$. This gives
\begin{align*}
0&=\eta(((a+m+n+bn)y(2n)-d_ny(n))(\bar{u}\otimes  v_{\mu}\otimes t^m))\\
&=((a+m+n+bn)y(2n)-d_ny(n))(\eta(\bar{u}\otimes  v_{\mu}\otimes t^m))\\
&=((a+m+n+bn)y(2n)-d_ny(n))(\sum_{i=1}^r P_i\bar{u}\otimes v_i\otimes t^m)\\
&=\sum_{i=1}^r ({\rm deg\,}P_i) P_i\bar{u}\otimes yv_i\otimes t^{m+2n},
\end{align*}
which together with the assumption $\mu\neq 0$ yields that ${\rm deg\,}P_i=0$ for all $i=1,\ldots, r$. Therefore,
\begin{align*}
\eta(\bar{u}\otimes v_{\mu}\otimes t^m)=\sum_{i=1}^r P_i\bar{u}\otimes v_i\otimes t^m\ \ \ \ {\rm with}\,\ {\rm deg\,}P_i=0, i=1,\ldots,r,\end{align*}
%For any $n\in\N$ with $a+m+nb\neq0$, the above along with $T(d_{n}(u\otimes  v_{\mu}\otimes t^m))=d_{n}(T(u\otimes  v_{\mu}\otimes t^m))$
%gives rise to
%\begin{align}\label{equ112}
%T(u\otimes v_{\mu}\otimes t^{m+n})=\sum_{i=1}^r P_iu\otimes v_i\otimes t^{m+n} \ \ \ \ {\rm with}\,\ {\rm deg\,}P_i=0, i=1,\ldots,r.
%\end{align}
which combined with
$$\eta(\mathfrak{g}_{\a}(n)(\bar{u}\otimes  v_{\mu}\otimes t^m))=\mathfrak{g}_{\a}(n)(\eta(\bar{u}\otimes  v_{\mu}\otimes t^m)),\ \ \ \forall \a\in\Delta_+, n\in\N$$
gives  $\sum_{i=1}^r P_i\bar{u}\otimes \mathfrak{g}_{\a}v_i\otimes t^{m+n}=0$. In view of the linear independence of $P_i\bar{u}, i=1,\ldots, r$, we obtain
\begin{align*}
\mathfrak{g}_{\a}v_i=0, \ \ \ \ i=1,\ldots,r, \,\forall \a\in\Delta_+.
\end{align*}
%Combining \eqref{equ111}, \eqref{equ112} with
%$$T(h(n)(u\otimes  v_{\mu}\otimes t^m))=h(n)(T(u\otimes  v_{\mu}\otimes t^m)),\ \ \ \ \forall h\in\mathfrak{h}, n\in\N\,\,{\rm with}\,\, a+m+nb\neq0,$$
%we know that $hv_i=\mu(h)v_i, i=1,\ldots,r$ for any $h\in\mathfrak{h}$, which  together with \eqref{equ113} gives $r=1$ and $v_{1}=c_1v_{\mu}$ for some $c_1\in\C$.
That is,  $\eta(\bar{u}\otimes v_{\mu}\otimes t^{m})=P\bar{u}\otimes v_{\mu}\otimes t^{m}$ with ${\rm deg\,}P=0$.
Since $\eta(h(\bar{u}\otimes  v_{\mu}\otimes t^m))=h(\eta(\bar{u}\otimes  v_{\mu}\otimes t^m))$, we have $hP\bar{u}=\lambda(h)P\bar{u}$ for any $h\in\mathfrak{h}$. This implies that $P\bar{u}\in\C \bar{u}$. Hence, there exist $c_m\in\C$ such that
$$\eta(\bar{u}\otimes  v_{\mu}\otimes t^m)=c_m(\bar{u}\otimes  v_{\mu}\otimes t^m),\ \ \ \ \forall m\in\Z.$$
It suffices to show that $c_m$  are complex constants independent of $m$ for $m\in\mathbb{Z}$. For any $n\in\N$, we have
\begin{align}\label{equation123}
&(a+m+bn)c_{m+n}\bar{u}\otimes v_{\mu}\otimes t^{m+n}\cr
=&\eta(d_n(\bar{u}\otimes v_{\mu}\otimes t^m))\cr
=&d_n\eta(\bar{u}\otimes v_{\mu}\otimes t^m)\cr
=&(a+m+bn)c_m\bar{u}\otimes v_{\mu}\otimes t^{m+n}.
\end{align}
%First consider the case $(a,b)\neq(0,0)$. In this case, for any $q\in\Z$, by taking a sufficiently large $n$ such that
%$$a+m+bn\neq0\quad {\rm and}\quad a+q+b(n-q+m)\neq0, $$
%one can deduce from \eqref{equation123} that $c_{m+n}=c_m$ and then $c_q=c_{(n-q+m)+q}=c_m$. In the remaining case $(a,b)=(0,0)$, similar arguments as above yields that
%\begin{align}\label{equation1234}
%c_p=c_q,\ \ \ \ \ \forall p,q\neq0.
%\end{align}
%Now take any $p\in\N$. It follows from
%\begin{align*}
%&pc_{0}u\otimes v_{\mu}\otimes t^{0}=\phi(d_p(u\otimes v_{\mu}\otimes t^{-p}))\nonumber\\
%&\ \ \ \ =d_p\phi(u\otimes v_{\mu}\otimes t^{-p})=pc_{-p}u\otimes v_{\mu}\otimes t^{0}
%\end{align*}
%that $c_0=c_{-p}$, which along with \eqref{equation1234} means that $c_m=c$ for all $m\in\Z$.
We divided that following discussion into two cases.

Case 1: $b\neq 0$.

Fix $m\in\mathbb{Z}$. There is at most one positive integer $n_0$ such that $a+m+bn_0=0$. It follows from (\ref{equation123}) that $c_m=c_{m+n}$ for any positive integer $n\neq n_0$. We can replace $m$ with $m+n_0$, so that $c_{m+n_0}=c_{m+n_0+n}$ for any positive integer $n\neq n_1$. This implies that $c_m=c_{m+n}$ for any $n\in\mathbb{N}$. Since $m\in\mathbb{Z}$ is arbitrary, we have $c_m=c_p$ for any $m, p\in\mathbb{Z}$.

Case 2: $b=0$.

If $a\neq 0$, then $a+m\neq 0$.  It follows from (\ref{equation123}) that $c_m=c_{m+n}$ for any $m\in\mathbb{Z}, n\in\mathbb{N}$. This implies that $c_m=c_p$ for any $m, p\in\mathbb{Z}$. If $a=0$,  It follows from (\ref{equation123}) that $c_m=c_{m+n}$ for any $0\neq m\in\mathbb{Z}, n\in\mathbb{N}$, which also yields that $c_m=c_p$ for any $m, p\in\mathbb{Z}$.

In conclusion, $c_m=c$ for all $m\in\mathbb{Z}$. Consequently, the desired result follows directly from Lemma \ref{lemma1}.
\end{proof}

\section{Irreducibility of tensor product weight modules}
In this section, we always assume $0\neq\mu\in\Lambda^+$. We shall obtain a necessary and sufficient condition for the tensor product modules $L_{\lambda, l, k, c}^{\mu, a,b}$ over $\mathfrak{L(g)}$ to be irreducible. We first give the following crucial lemma.
\begin{lemm}\label{lemma2}
For any nonzero $\mathfrak{L(g)}$-submodule $W$ of $L_{\lambda, l, k, c}^{\mu, a,b}$,  we have $L(\lambda, l, k, c)\otimes L(\mu)\otimes t^n\subseteq W$ for all sufficiently large $n\in\Z$.
\end{lemm}
\begin{proof}
Denote $L=L_{\lambda, l, k, c}^{\mu, a,b}$ and $L_n=L(\lambda, l, k, c)\otimes L(\mu)\otimes t^n$ for convenience. %And elements in $L_n$ are called homogeneous.
Since $L$ is a weight module, there exist subspaces $W_n\subseteq L(\lambda, l, k, c)\otimes L(\mu)$ for all $n\in\Z$ such that $W=\oplus_{n\in\Z}W_n\otimes t^n$. Choose any $n'$ such that $W_{n'}\neq0$ and take a nonzero $w\in W_{n'}$. Write
\begin{align}\label{cbn}w=\sum_{i=1}^r P_i\bar{u}\otimes v_i\end{align}
with $0\neq v_i\in L(\mu)$ for $i=1,\cdots, r$, and homogeneous elements $P_1,\ldots, P_r\in U(\mathfrak{L(g)}_{-})$ such that $0\geq {\rm deg\,}P_1\geq \cdots \geq {\rm deg\,}P_r$ and $P_1\bar{u},\ldots, P_r\bar{u}$ are linearly independent. Obviously, there exists $q\in\Z_+$ such that $\mathfrak{g}(m)(P_i\bar{u})=0$ and $d_m(P_i\bar{u})=0, m\geq q$ for $i=1,\ldots,r$.
%Recall that $L(\mu)$ is an irreducible $\mathfrak{g}$-module with the highest weight $\mu\in\Lambda^{+}$ and for any weight $\eta,\varepsilon\in\mathfrak{h}^*$, we say $\varepsilon\preceq \eta$ if $\eta-\varepsilon$ is the sum of nonnegative roots. For any element $\tilde{v}=\tilde{v}_1+\cdots+\tilde{v}_s\in L(\mu)$ such that $\tilde{v}_i's$ are nonzero weight vectors of distinct weights, denote ${\rm wt\,}(\tilde{v})={\rm min\,}\{{\rm wt\,}(\tilde{v}_i), i=1,\ldots,s\}$.
If there is $v_i$, without loss of generality we assume $v_1$, in the expression \eqref{cbn} so that $v_1\notin \mathbb{C} \bar{u}$, then $v_1$ can not be annihilated by the whole positive part of $\mathfrak{g}$, that is, $xv_1\neq0$ for some $x\in\mathfrak{n}_{+}$. Then we have
$$x(q)(w\otimes t^{n'})=x(q)(\sum_{i=1}^r P_i\bar{u}\otimes v_i\otimes t^{n'})=\sum_{i=1}^rP_i\bar{u}\otimes xv_i\otimes t^{n'+q},$$
which is nonzero by the linear independence of $P_i\bar{u}, i=1,\ldots,r$. Repeating this process, we can obtain a nonzero element $w'=\sum_{i=1}^{r} P_i\bar{u}\otimes v_i'\in W_{n''}$ with either $0\neq v_i'\in\mathbb{C}\bar{u}$ or $v_i'=0$ for any $i=1,\ldots,r$. Then  $w'=\alpha \otimes v_{\mu}\in W_{n''}$ for some nonzero $\alpha\in L(\lambda, k, c, l)$. In fact, $\alpha$ is a linear combination of $P_1\bar{u},\ldots, P_r\bar{u}$. By replacing $w$ with $w'$ and $n'$ with $n''$, we  may assume that $w=\alpha\otimes v_{\mu}\in W_{n'}$.
\begin{clai}
$Q_1\bar{u}\otimes L(\mu)\otimes t^n\subseteq W$ for all sufficiently large $n$, where  $Q_1$ is a homogeneous element in $U(\mathfrak{L(g)}_{-})$.
\end{clai}
Since  $L(\mu)$ is a finite-dimensional $\mathfrak{g}$-module with the highest weight vector $v_{\mu}$, $L(\mu)$ has a basis consisting of only finitely many vectors in the form $X_1v_{\mu},\ldots, X_sv_{\mu}$, where $X_i\in U(\mathfrak{n}_{-})$. Combining this with
$$x(m)(\alpha\otimes v_{\mu}\otimes t^{n'})=\alpha\otimes xv_{\mu}\otimes t^{n'+m},\ \ \ \ \forall m\geq q, \,\,x\in \mathfrak{n}_{-}\cup  \mathfrak{h},$$
we have
\begin{equation}\label{cv}\alpha\otimes L(\mu)\otimes t^{n}\subseteq W \ \ \ \ {\rm for \,\,all \,\,sufficiently\,\,large} \,\,n. \end{equation}
Among all nonzero $\a=\sum_{i=1}^{r}Q_i\bar{u}\in L(\lambda, l, k, c)$ satisfying \eqref{cv}, where $Q_1,\ldots, Q_r\in U(\mathfrak{L(g)}_{-})$ are homogeneous elements with $0\geq {\rm deg\,}Q_1> \cdots > {\rm deg\,}Q_r$, we can choose $\a$ with $r$ being minimal. For $m,n,p>-{\rm deg\,}Q_r$, we have
\begin{align*}
d_md_n(\a\otimes L(\mu) \otimes t^p)&=\sum_{i=1}^r (a+p+nb-{\rm deg\,}Q_i)(a+p+n+mb-{\rm deg\,}Q_i)Q_iu\otimes L(\mu)\otimes t^{p+m+n},\\
d_{m+n}(\a\otimes L(\mu) \otimes t^p)&=\sum_{i=1}^r (a+p-{\rm deg\,}Q_i+b(m+n))Q_iu\otimes L(\mu)\otimes t^{p+m+n}.
\end{align*}
Then, for any $1\leq j\leq r$, we consider
\begin{align*}
&(a+p-{\rm deg\,}Q_j+b(m+n))d_md_n(\a\otimes L(\mu)\otimes t^p)-(a+p+nb-{\rm deg\,}Q_j)\\
&\ \ \ \cdot (a+p+n+mb-{\rm deg\,}Q_j)d_{m+n}(\a\otimes L(\mu)\otimes t^p)\\
&=\sum_{i=1}^r\big((a+p-{\rm deg\,}Q_j+b(m+n))(a+p+nb-{\rm deg\,}Q_i)(a+p+n+mb-{\rm deg\,}Q_i)\\
&\ \ \ - (a+p+nb-{\rm deg\,}Q_j)(a+p+n+mb-{\rm deg\,}Q_j)(a+p-{\rm deg\,}Q_i+b(m+n)\big)\\
&\ \ \ \cdot Q_iu\otimes L(\mu)\otimes t^{p+m+n}\\
&=\sum_{i=1}^r({\rm deg\,}Q_j-{\rm deg\,}Q_i)\big(b^2(m^2+mn+n^2)+bmn+b(2a+2p-{\rm deg\,}Q_i-{\rm deg\,}Q_j)(m+n)\\
&\ \ \ + (a+p-{\rm deg\,}Q_i)(a+p-{\rm deg\,}Q_j)\big)Q_iu\otimes L(\mu)\otimes t^{p+m+n}\subseteq W.
\end{align*}
If $r>1$, by the minimality of $r$, then we have
\begin{align*}
&b^2(m^2+mn+n^2)+bmn+b(2a+2p-{\rm deg\,}Q_{i}-{\rm deg\,}Q_{j})(m+n)\\
&\ \ \ +(a+p-{\rm deg\,}Q_{i})(a+p-{\rm deg\,}Q_{j})=0, \ \ \ \ \forall i\neq j, m,n,p>-{\rm deg\,}Q_r,
\end{align*}
from which one can deduce that $b=0$ and $(a+p-{\rm deg\,}Q_{i})(a+p-{\rm deg\,}Q_{j})=0$, which is impossible. Consequently, $r=1$ and the claim follows.
%\begin{clai}\label{claim2}
%$L(\lambda, l, k, c)\otimes L(\mu)\otimes t^n\subseteq W$ for all sufficiently large $n$.
%\end{clai}

Now  we suppose that $Q_1\bar{u}\otimes L(\mu)\otimes t^m\subseteq W$ for all $m\geq n_1$ with $n_1\in\Z$ is fixed.
%For any $i\in\N$ and $m\geq n_1$, we have
%$$d_i(Q_1u\otimes L(\mu)\otimes t^m)=(d_i+a+m-{\rm deg\,}Q_1+ib)Q_1u\otimes L(\mu)\otimes t^{m+i}\subseteq W,$$
%which indicates $d_iQ_1u\otimes L(\mu)\otimes t^{m+i}\subseteq W$. Inductively, we can show that $$d_{i_{s}}\cdots d_{i_1}Q_1u\otimes L(\mu)\otimes t^{m+i_1+\cdots+i_{s}}\subseteq W,\ \ \ \ \forall m\geq n_1, i_1,\ldots,i_s\in\N.$$
%Taking any $j_{1},\ldots,j_{t}\in\N$, $j_{t+1},\ldots,j_{t+k}\in\Z_+$, $x_{1},\ldots, x_{t}\in \mathfrak{n}_{-}\oplus\mathfrak{h}, x_{t+1}, \ldots, x_{t+k}\in \mathfrak{n}_{+}$
%%$$x_{s_1+s_2+s_3}(j_{s_1+s_2+s_3})\cdots x_{s_1+s_2+1}(j_{s_1+s_2+1})x_{s_1+s_2}(j_{s_1+s_2})\cdots x_{s_{1}+1}(j_{s_{1}+1})  d_{j_{s_{1}}}\cdots d_{j_1}Q_1u=u.$$
%and applying $x_{1}(j_{1}),\ldots,x_{t}(j_{t}),\ldots, x_{t+k}(j_{t+k})$ in turn to the above element, we see that
%\begin{equation}\label{equation1}x_{t+k}(j_{t+k})\cdots x_{1}(j_{1})d_{i_{s}}\cdots d_{i_1}Q_1u\otimes L(\mu)\otimes t^{m+i_1+\cdots+i_{s}+j_{1}+\cdots+j_{t+k}}\subseteq W,\ \ \ \ \forall m\geq n_1.\end{equation}
There exists some homogeneous $x\in U(\mathfrak{L(g)})_{-{\rm deg\,}Q_1}$ such that $x(Q_1\bar{u})=u$ since $L(\lambda, l, k, c)$ is an irreducible $\mathfrak{L(g)}$-module. This along with \eqref{equ1} forces $\bar{u}\otimes L(\mu)\otimes t^{n}\subseteq W$ for all sufficiently large $n$. Proceeding by downward induction on ${\rm deg\,}P, P\in U(\mathfrak{L(g)}_{-})$, one can easily deduce that $P\bar{u}\otimes L(\mu)\otimes t^{n}\subseteq W$
for all homogeneous  $P\in U(\mathfrak{L(g)}_{-})$ and sufficiently large $n$. We complete the proof.
\end{proof}
From the proof of the above lemma, we have the following  result.
\begin{coro}\label{coro1}
Let $W$ be a submodule of $L_{\lambda, l, k, c}^{\mu, a,b}$ and $n_0\in\Z$. If $\bar{u}\otimes L(\mu)\otimes t^n\subseteq W$ for all $n\geq n_0$, then $L(\lambda, l, k, c)\otimes L(\mu)\otimes t^n\subseteq W$ for all $n\geq n_0$.
\end{coro}
Notice that $U(\mathfrak{L(g)}_{-})= U(\mathfrak{g}\otimes t^{-1}\C[t^{-1}])\otimes U(\mathcal{V}_{-})\otimes U(\mathfrak{n}_{-})$. Fix any $a,b\in\C$ with  $0\leq{\rm Re\,}a<1$. For any $n\in\Z$, we have a linear map $\psi_n$ from $U(\mathcal{V}_{-})$ to $\C$ defined by
$$\psi_n(d_{-k_{r}}\cdots d_{-k_{1}})=\prod_{j=1}^r(k_{j}b-a-n-\sum_{i=1}^{j}k_i), \ \ \ \ \forall d_{-k_{r}}\cdots d_{-k_{1}}\in U(\mathcal{V}_{-}).$$
For more details of this map see \cite{CGZ}. Let $\{x_1,\ldots,x_d\}$ be a basis of $\mathfrak{g}$. Denote by  $\mathbb{M}$  the set of all non-decreasing sequences $\mathbf{i}:=(i_1, i_2, \ldots)$ of finitely many positive integers. By the $\mathrm{PBW}$ Theorem, every element of $U(\mathfrak{g}\otimes t^{-1}\C[t^{-1}])$ can be uniquely written in the following form
$$x=\sum_{( \mathbf{i},\mathbf{j})\in\mathbb{M}^2} a_{\mathbf{i},\mathbf{j}} x_{j_s}(-i_s)\cdots x_{j_2}(-i_2)x_{j_1}(-i_1)\  \ \ \ {\rm with\,}a_{\mathbf{i},\mathbf{j}}\in\C.$$
For any given $x\in U(\mathfrak{g}\otimes t^{-1}\C[t^{-1}])$, when we omit the index $\mathbf{i}$, we get the associated
element $\tilde{x}=\sum_{\mathbf{j}\in\mathbb{M}} a_{\mathbf{i},\mathbf{j}} x_{j_s}\cdots x_{j_2}x_{j_1}\in U(\mathfrak{g})$. Put $L(\lambda, l, k, c)_0=\{v\in L(\lambda, l, k, c)\mid d_0v=lv\}$. Now for any $n\in\Z$, we can define a linear map
\begin{align*}
\varphi_n: \ \ \ \ U(\mathfrak{L(g)}_{-})\otimes L(\mu)&\rightarrow L(\lambda, l, k, c)_{0}\otimes L(\mu)\\
xyz\otimes v &\mapsto \psi_n(y)z\bar{u}\otimes (\tilde{x}\circ v)
\end{align*}
for all $x\in U(\mathfrak{g}\otimes t^{-1}\C[t^{-1}]), y\in U(\mathcal{V}_{-}), z\in U(\mathfrak{n}_{-})$ and $v\in L(\mu)$, where the action $\circ$ is defined as follows:
$$1\circ v=v,\ \ \ \ (g_1g_2\cdots g_s)\circ v=(-1)^s g_s\cdots g_2 g_1 v,\ \ \ \ \forall g_i\in \mathfrak{g}.$$

For the highest weight vector $\bar{u}$ of $L(\lambda, l, k, c)$, the annihilator of $\bar{u}$ is defined as $${\rm Ann\,}_{U(\mathfrak{L(g)}_{-})}(\bar{u})=\{P\in U(\mathfrak{L(g)}_{-})\mid P\bar{u}=0\}.$$

We are now in a position to present the following sufficient and necessary condition for a tensor product module to be irreducible.
\begin{theo}\label{theorem2}
$L_{\lambda, l, k, c}^{\mu, a,b}$  is irreducible if and only if $\varphi_n( {\rm Ann\,}_{U(\mathfrak{L(g)}_{-})}(\bar{u})\otimes L(\mu))=L(\lambda, l, k, c)_{0}\otimes L(\mu)$ for all  $n\in\Z$.
\end{theo}
\begin{proof}
Let $W=\sum_{n\in\Z}W_n\otimes t^n$ be a nonzero submodule of  $L_{\lambda, l, k, c}^{\mu, a,b}$. Thanks to Lemma \ref{lemma2}, there exists $N\in\Z$ such that $L(\lambda, l, k, c)\otimes L(\mu)\otimes t^i\subseteq W$ for all $i\geq N+1$.
\begin{clai}\label{cvbn}
For any $v\in L(\mu), x\in U(\mathfrak{g}\otimes t^{-1}\C[t^{-1}]), y\in U(\mathcal{V}_{-})$ and $z\in U(\mathfrak{n}_{-})$, we have $xyz\bar{u}\otimes v\equiv \psi_{N}(y)z\bar{u}\otimes (\tilde{x}\circ v)=\varphi_{N}(xyz\otimes v)\,\,{\rm mod\,}W_{N}$.
\end{clai}
By linearity it is sufficient to prove the claim for $x=x_{j_s}(-i_s)\cdots x_{j_1}(-i_1)\in U(\mathfrak{g}\otimes t^{-1}\C[t^{-1}]), y=d_{-k_{r}}\cdots d_{-k_{1}}\in U(\mathcal{V}_{-})$ and $z\in U(\mathfrak{n}_{-})$ with $x_{j_1},\ldots,x_{j_s}\in \mathfrak{g}$ and $i_1,\ldots,i_s\in\N, k_1\ldots,k_r\in\N$. We proceed by induction on $s$. If $s=0$, we again do it by induction on $r$. The case $r=0$ is clear. Assume that the result is true for $y=d_{-k_{r}}\cdots d_{-k_{1}}\in U(\mathcal{V}_{-}), z\in U(\mathfrak{n}_{-})$ and $v\in L(\mu)$. Then from
\begin{align*}
&d_{-k_{r+1}}(d_{-k_{r}}\cdots d_{-k_{1}}z\bar{u}\otimes v\otimes t^{N+k_{r+1}})\\
=&(d_{-k_{r+1}}+a+N+\sum_{j=1}^{r+1}k_j-bk_{r+1})(d_{-k_{r}}\cdots d_{-k_{1}}z\bar{u})\otimes v\otimes t^{N}\in W,\end{align*}
it follows immediately that
\begin{align*}
&(d_{-k_{r+1}}d_{-k_{r}}\cdots d_{-k_{1}}z\bar{u})\otimes v\\
\equiv &-(a+N+\sum_{j=1}^{r+1}k_j-bk_{r+1})(d_{-k_{r}}\cdots d_{-k_{1}}z\bar{u})\otimes v\\
\equiv &-(a+N+\sum_{j=1}^{r+1}k_j-bk_{r+1})\psi_{N}(d_{-k_{r}}\cdots d_{-k_{1}})z\bar{u}\otimes v\\
\equiv & \psi_{N}(d_{-k_{r+1}}d_{-k_{r}}\cdots d_{-k_{1}})z\bar{u}\otimes v\,\,{\rm mod\,}W_{N}.
\end{align*}
Thus, the case $s=0$ holds. Now suppose that the result is true for  $x=x_{j_s}(-i_s)\cdots x_{j_1}(-i_1)$, $y=d_{-k_{r}}\cdots d_{-k_{1}},z\in U(\mathfrak{n}_{-})$ and $v\in L(\mu)$. Then for any $ x_{j_{s+1}}\in\mathfrak{g}, i_{s+1}\in\N$, we have
\begin{align*}
&x_{j_{s+1}}(-i_{s+1})\big(x_{j_s}(-i_s)\cdots x_{j_1}(-i_1)d_{-k_{r}}\cdots d_{-k_{1}}z\bar{u}\otimes v\otimes t^{N+i_{s+1}}\big)\\
=&x_{j_{s+1}}(-i_{s+1})x_{j_s}(-i_s)\cdots x_{j_1}(-i_1)d_{-k_{r}}\cdots d_{-k_{1}}z\bar{u}\otimes v\otimes t^{N}\\
&+x_s(-i_s)\cdots x_1(-i_1)d_{-k_{r}}\cdots d_{-k_{1}}z\bar{u}\otimes x_{j_{s+1}}v\otimes t^{N}\in W,
\end{align*}
forcing
\begin{align*}
&x_{j_{s+1}}(-i_{s+1})x_{j_s}(-i_s)\cdots x_{j_1}(-i_1)d_{-k_{r}}\cdots d_{-k_{1}}z\bar{u}\otimes v\\
\equiv &-x_{j_s}(-i_s)\cdots x_{j_1}(-i_1)d_{-k_{r}}\cdots d_{-k_{1}}z\bar{u}\otimes x_{j_{s+1}} v\,\,{\rm mod\,}W_{N}\\
\equiv &-\psi_{N}(d_{-k_{r}}\cdots d_{-k_{1}})z\bar{u}\otimes ((x_{j_s}\cdots x_{j_1})\circ x_{j_{s+1}} v)\,\,{\rm mod\,}W_{N}\\
\equiv & \psi_{N}(d_{-k_{r}}\cdots d_{-k_{1}})z\bar{u}\otimes ((x_{j_{s+1}}x_{j_s}\cdots x_{j_1})\circ  v)\,\,{\rm mod\,}W_{N}.
\end{align*}
Hence the claim follows by induction.

First we assume that $\varphi_n( {\rm Ann\,}_{U(\mathfrak{L(g)}_{-})}(\bar{u})\otimes L(\mu))=L(\lambda, l, k, c)_{0}\otimes L(\mu)$ for all $n\in\Z$. Then for any $v\in L(\mu)$, there exist $x_i\in {\rm Ann\,}_{U(\mathfrak{L(g)}_{-})}(u)$ and $v_i\in L(\mu)$ such that $\varphi_N(\sum_{i=1}^s x_i\otimes v_i)=\bar{u}\otimes v$. By Claim \ref{cvbn}, we have
$$\bar{u}\otimes v=\varphi_{N}(\sum_{i=1}^s x_i\otimes v_i)\equiv \sum_{i=1}^s x_i\bar{u}\otimes v_i=0\,\,{\rm mod\,}W_{N}.$$
That is $\bar{u}\otimes L(\mu)\subseteq W_{N}$. We conclude that $L(\lambda, l, k, c)\otimes L(\mu)=W_{N}$ from Corollary \ref{coro1}. Inductively, we get $W_i=L(\lambda, l, k, c)\otimes L(\mu)$ for all $i\in\Z$, i.e., $W=L_{\lambda, l, k, c}^{\mu, a,b}$ and hence $L_{\lambda, l, k, c}^{\mu, a,b}$ is irreducible.

Conversely, suppose that $\varphi_{n_0}( {\rm Ann\,}_{U(\mathfrak{L(g)}_{-})}(\bar{u})\otimes L(\mu))\neq L(\lambda, l, k, c)_{0}\otimes L(\mu)$ for some $n_0\in\Z$. Then we have an induced map
\begin{align*}
\overline{\varphi}_{n_0}: \ \ \ \ U(\mathfrak{L(g)}_{-})\otimes L(\mu)&\rightarrow (L(\lambda, l, k, c)_{0}\otimes L(\mu))/\varphi_{n_0}({\rm Ann\,}_{U(\mathfrak{L(g)}_{-})}(\bar{u})\otimes L(\mu))\\
P\otimes v&\mapsto\overline{\varphi_{n_0}(P\otimes v)}
\end{align*}
for all $P\in U(\mathfrak{L(g)}_{-})$ and $v\in L(\mu)$. It is apparent that $\overline{\varphi}_{n_0}({\rm Ann\,}_{U(\mathfrak{L(g)}_{-})}(\bar{u})\otimes L(\mu))=0$. Observing that $L(\lambda, l, k, c)=U(\mathfrak{L(g)}_{-})\bar{u}\cong U(\mathfrak{L(g)}_{-})/{\rm Ann\,}_{U(\mathfrak{L(g)}_{-})}(\bar{u})$, then we obtain the induced linear map
\begin{align*}
\widehat{\varphi}_{n_0}: \ \ \ \ L(\lambda, l, k, c)\otimes L(\mu)&\rightarrow (L(\lambda, l, k, c)_{0}\otimes L(\mu))/\varphi_{n_0}({\rm Ann\,}_{U(\mathfrak{L(g)}_{-})}(\bar{u})\otimes L(\mu))\\
P\bar{u}\otimes v&\mapsto \overline{\varphi_{n_0}(P\otimes v)}
\end{align*}
for all $P\in U(\mathfrak{L(g)}_{-})$ and $v\in L(\mu)$. Let $W^{(n_0)}$ be the submodule of  $L_{\lambda, l, k, c}^{\mu, a,b}$ generated by  $L(\lambda, l, k, c)\otimes L(\mu)\otimes t^i,i\geq n_0+1$. We shall show that  $W^{(n_0)}$ is a proper submodule of $L_{\lambda, l, k, c}^{\mu, a,b}$. According to the $\mathrm{PBW}$ Theorem, the weight space of  $W^{(n_0)}$ with respect to the action of $d_0$ of weight $a+l+n_0$, denoted by $W^{(n_0)}_{n_0}$, is a sum of elements of the form
$$G(-i)(xyz\bar{u}\otimes v\otimes t^{n_0+i})=\big(G(-i)xyz \bar{u}\otimes v+xyz u\otimes Gv\big)\otimes t^{n_0}$$
and
$$d_{-i}(xyz\bar{u}\otimes v\otimes t^{n_0+i})=(d_{-i}+a+n_0+i-{\rm deg\,}(xyz)-bi)xyz\bar{ u}\otimes v\otimes t^{n_0},$$
where $G\in\mathfrak{g},x\in U(\mathfrak{g}\otimes t^{-1}\C[t^{-1}]), y\in U(\mathcal{V}_{-}), z\in U(\mathfrak{n}_{-}),  v\in L(\mu), i\in\N$ and  $x,y$ are homogeneous elements. Explicit calculations show that
\begin{align*}
&\widehat{\varphi}_{n_0}\big(G(-i)xyz \bar{u}\otimes v+xyz \bar{u}\otimes Gv\big)\\
=&\overline{\varphi_{n_0}\big(G(-i)xyz\otimes v+xyz\otimes Gv\big)}\\
=&\psi_{n_0}(y)\overline{z\bar{u}\otimes (G\tilde{x}\circ v)+z\bar{u}\otimes (\tilde{x}\circ Gv)}\\
=&0
\end{align*}
and
\begin{align*}
&\widehat{\varphi}_{n_0}\big((d_{-i}+a+n_0+i-{\rm deg\,}(xyz)-bi)xyz \bar{u}\otimes v\big)\\
=&\overline{\varphi_{n_0}\big((d_{-i}+a+n_0+i-{\rm deg\,}(xyz)-bi)xyz \otimes v\big)}\\
=&\overline{\varphi_{n_0}(d_{-i}xyz \otimes v)+(a+n_0+i-{\rm deg\,}(xyz)-bi)\psi_{n_0}(y)z\bar{u} \otimes (x\circ v)}\\
%&=\overline{\varphi_{n_0}\big((x(\mathbf{i})d_{-i}+{\rm deg\,}(x(\mathbf{i}))x(\mathbf{i}'))yz \otimes v\big)}+(a+n_0+i-{\rm deg\,}(x(\mathbf{i})yz)-bi)\\
%&\ \ \ \cdot\psi_{n_0}(y)\overline{zu \otimes (x\circ v)}\\
=&\psi_{n_0}(d_{-i}y)\overline{z\bar{u} \otimes (\tilde{x}\circ v)}+{\rm deg\,}(x)\psi_{n_0}(y)\overline{z\bar{u} \otimes (\tilde{x}\circ v)}\\
&+(a+n_0+i-{\rm deg\,}(xyz)-bi)\psi_{n_0}(y)\overline{z\bar{u} \otimes (\tilde{x}\circ v)}\\
=&0.
\end{align*}
Note that the last equality follows directly from $\psi_{n_0}(d_{-i}y)=(bi-a-n_0+{\rm deg\,}(y)-i)\psi_{n_0}(y)$ and  ${\rm deg\,}(xyz)={\rm deg\,}(x)+{\rm deg\,}(y)$. As a result, we have $\widehat{\varphi}_{n_0}(W^{(n_0)}_{n_0})=0$, or equivalently, $W^{(n_0)}_{n_0}\subseteq {\rm Ker\,}(\widehat{\varphi}_{n_0})\neq L(\lambda, l, k, c)\otimes L(\mu)$ since $\widehat{\varphi}_{n_0}$ is surjective and nonzero. Thus, $W^{({n_0})}$ is a proper submodule of $L_{\lambda, l, k, c}^{\mu, a,b}$, that is, $L_{\lambda, l, k, c}^{\mu, a,b}$ is reducible, completing the proof.
\end{proof}
%\begin{rema}
%If the unique maximal submodule $J(\lambda, l, k, c)$ is generated by all singular vectors besides $\mathbb{C}u$, then we can give the description of the annihilator ${\rm Ann\,}_{U(\mathfrak{L(g)}_{-})}(u)$ by the structure of singular vectors given in Lemma \ref{coro-2}. More precisely,
%$${\rm Ann\,}_{U(\mathfrak{L(g)}_{-})}(\bar{u})=\{E\in U(\mathfrak{L(g)}_{-})\mathfrak{L(g)}_{-}\mid Eu\,\text{is\,\,a\,\,singular\,\,vector}\}.$$
%\end{rema}

From  \cite[Corollary 1]{K}, we see that  there are a lot of irreducible highest weight modules $L(\lambda, l, k, c)$ that are Verma modules over $\mathfrak{L(g)}$. Meanwhile, for these irreducible Verma modules, we know that ${\rm Ann\,}_{U(\mathfrak{L(g)}_{-})}(\bar{u})=0$.
As a direct consequence of Theorem \ref{theorem2}, we have the following.
\begin{coro}\label{last cor}
Let $L(\lambda, l, k, c)$ be an irreducible Verma module over $\mathfrak{L(g)}$. Then $L_{\lambda, l, k, c}^{\mu, a,b}$ is always reducible.
\end{coro}

We remark  that the condition in Theorem \ref{theorem2} is not easy to verify since we generally do not know what ${\rm Ann\,}_{U(\mathfrak{L(g)}_{-})}(\bar{u})$ is for an irreducible highest $\mathfrak{L(g)}$-module $L_{\lambda, l, k, c}^{\mu, a,b}$. Fortunately, we do have an explicit description of ${\rm Ann\,}_{U(\mathfrak{L(g)}_{-})}(\bar{u})$  under certain situations (see the forthcoming Remark \ref{rem for Ann}).

The following are standard notions from \cite{KV}. Let $\alpha_i,i=1,\ldots,r$ be the simple roots of $\mathfrak{g}$, $\check{\alpha_i}$ be the corresponding coroots, and $\delta$ be the standard imaginary root. Put $\alpha_0=\delta-\theta$. Then $\alpha_i,i=0,\ldots,r$ are the simple roots of  $\tilde{\mathfrak{g}}$. Take $e_{\theta}\in \mathfrak{g}_{\theta}$. Clearly, $f_0=e_{\theta}(-1)$ and $\check{\alpha_0}=\mathbf{k}-\check{\theta}$. As is well-known, any weight  $\Lambda\in(\tilde{\mathfrak{h}}+\C d_0)^{*}$ is of the form $\Lambda=\lambda+k \Lambda_0$, where $\lambda\in \mathfrak{h}^{*}, k\in\C$ and $\Lambda_0$ is the root satisfies $\langle\Lambda_0,\check{\alpha_i}\rangle=\delta_{i,0}, \langle\Lambda_0,d_0\rangle=0$. It is worth noting that $\Lambda=\lambda+k \Lambda_0$ is dominant if and only if $\lambda$ is a dominant weight relative to $\mathfrak{g}$ and $k\in\N$ such that $k\geq \langle\lambda,\check{\theta}\rangle$.
%provided that  $k\neq -g$ and the highest weight of $\tilde{\mathfrak{g}}$ is dominant.

%Now we assemble a few known results about Verma modules $M(\lambda, l, k, c), M_{\mathcal{V}}(l, c)$ and $M_{\tilde{\mathfrak{g}}}(\lambda, k)$.
We now recall and establish several auxiliary results.
\begin{lemm}\label{coro-21}(cf.~\cite{A,FF,KV}) The following statements hold.
\begin{itemize}
\item [(1)] There exist two homogeneous  elements $$F_1^{(l, c)}(d_{-i_{t}}\cdots d_{-i_{1}}), F_2^{(l, c)}(d_{-j_{s}}\cdots d_{-j_{1}})\in U(\mathcal{V}_{-})$$  such that
$$U(\mathcal{V}_{-})F_1^{(l, c)}(d_{-i_{t}}\cdots d_{-i_{1}})u_{l, c}+U(\mathcal{V}_{-})F_2^{(l, c)}(d_{-j_{s}}\cdots d_{-j_{1}})u_{l, c}$$ is the unique maximal proper submodule of $M_{\mathcal{V}}(l, c)$.  %where $F_1^{(l, c)}, F_2^{(l, c)}$ are allowed to be zero;
\item [(2)] If  $\lambda+k \Lambda_0$ is a dominant weight, then
$$U(\tilde{\mathfrak{n}}_{-})f_0^{\langle\lambda+k\Lambda_0,\check{\alpha_0}\rangle+1}u_{\lambda, k}+\cdots+U(\tilde{\mathfrak{n}}_{-})f_r^{\langle\lambda+k\Lambda_0,\check{\alpha_r}\rangle+1}u_{\lambda, k}$$
is the unique maximal proper submodule of $M_{\tilde{\mathfrak{g}}}(\lambda, k)$.
\end{itemize}

\end{lemm}

\begin{lemm}\label{coro-2}(cf.~\cite{K})
Keep notations as before. Assume that $k\neq -g$. Then we have $\mathfrak{L(g)}$-module isomorphisms
\begin{align*}
\pi_1&:  M(\lambda, l, k, c)\cong M_{\tilde{\mathfrak{g}}}^{\mathfrak{L(g)}}(\lambda, k)\otimes M_{\mathcal{V}}^{\mathfrak{L(g)}}(l+\frac{c_{\lambda}}{2(k+g)}, c-\frac{k({\rm dim \,}\mathfrak{g})}{k+g}),\\
\pi_2&:  L(\lambda, l, k, c)\cong L_{\tilde{\mathfrak{g}}}^{\mathfrak{L(g)}}(\lambda, k)\otimes L_{\mathcal{V}}^{\mathfrak{L(g)}}(l+\frac{c_{\lambda}}{2(k+g)}, c-\frac{k({\rm dim \,}\mathfrak{g})}{k+g}).
\end{align*}
%Moreover, we have
%\begin{itemize}
%\item [(1)]
%  $F_1^{(l, c)}(D_{-i_{t}}\cdots D_{-i_{1}})u, F_2^{(l, c)}(D_{-j_{s}}\cdots D_{-j_{1}})u$ are singular vectors in $$M(\lambda, l-\frac{c_{\lambda}}{2(k+g)}, k, c-\frac{k({\rm dim \,}\mathfrak{g})}{k+g}).$$
%  \item [(2)]Let $(\lambda+\rho\mid \alpha)+m(k+g)-\frac{1}{2}n(\alpha\mid \alpha)=0$ for some $m\in\Z_+$ and $n\in\N$ such that $\alpha+m \delta$ lies in the positive roots of  $\mathfrak{L(g)}$. One knows that there exists a singular vector $Eu_{\lambda, k}$ for $\tilde{\mathfrak{g}}$ in $\big(M_{\tilde{\mathfrak{g}}}(\lambda, k)\big)_{\lambda-n(\alpha+m\delta)}$. Then $Eu$ is a singular vector for $\mathfrak{L(g)}$ in $M(\lambda, l, k, c)$.
%\end{itemize}
\end{lemm}
From now on, we assume that $\lambda+k \Lambda_0\in(\tilde{\mathfrak{h}}+\C d_0)^{*}$ is dominant, which implies $k\neq -g$. Set $D_n=-\frac{1}{k+g}T_n+d_n$ for any $n\in\Z$. It is easy to see that $L(\lambda, l, k, c)_{0}=L(\lambda)$ is just the finite-dimensional irreducible $\mathfrak{g}$-module with highest weight $\lambda$ , and we still use $\bar{u}$ to denote the highest weight vector of $L(\lambda)$ .
\begin{lemm}\label{coro-3}
The unique proper maximal submodule $J(\lambda, l, k, c)$ of  $M(\lambda, l, k, c)$ is generated by some singular vectors besides $\mathbb{C}u$.
\end{lemm}
\begin{proof}
Denote by $l'=l+\frac{c_{\lambda}}{2(k+g)}$ and $c'=c-\frac{k({\rm dim \,}\mathfrak{g})}{k+g}$ for short. From Lemma \ref{coro-21}, we know that the unique  maximal proper submodule $J_{\tilde{\mathfrak{g}}}(\lambda, k)$ of $M_{\tilde{\mathfrak{g}}}(\lambda, k)$ (resp. $J_{\mathcal{V}}(l', c')$ of $M_{\mathcal{V}}(l', c')$)
is generated by  singular vectors
\begin{align*}
f_i^{\langle\lambda+k\Lambda_0,\check{\alpha_i}\rangle+1}u_{\lambda, k},\,\,i=0,\ldots,r \,\,\big({\rm resp.\,} F_1^{(l', c')}(d_{-i_{t}}\cdots d_{-i_{1}})u_{l', c'}, F_2^{(l', c')}(d_{-j_{s}}\cdots d_{-j_{1}})u_{l', c'}\big).
 \end{align*}
We  use $u_{\lambda, -\frac{c_{\lambda}}{2(k+g)}, k, \frac{k({\rm dim \,}\mathfrak{g})}{k+g}}, u_{0, l', 0, c'}$ to denote the highest weight vector of $M_{\tilde{\mathfrak{g}}}^{\mathfrak{L(g)}}(\lambda, k)$ and $M_{\mathcal{V}}^{\mathfrak{L(g)}}(l', c')$, respectively. A quick calculation shows that the  proper submodule $J_{\tilde{\mathfrak{g}}}^{\mathfrak{L(g)}}(\lambda, k)$ of $M_{\tilde{\mathfrak{g}}}^{\mathfrak{L(g)}}(\lambda, k)$ (resp. $J_{\mathcal{V}}^{\mathfrak{L(g)}}(l', c')$ of $M_{\mathcal{V}}^{\mathfrak{L(g)}}(l', c')$)
is also generated by singular  vectors
$$f_i^{\langle\lambda+k\Lambda_0,\check{\alpha_i}\rangle+1}u_{\lambda, -\frac{c_{\lambda}}{2(k+g)}, k, \frac{k({\rm dim \,}\mathfrak{g})}{k+g}},\,\, i=0,\ldots,r$$ (resp. $F_1^{(l', c')}(d_{-i_{t}}\cdots d_{-i_{1}})u_{0, l', 0, c'}, F_2^{(l', c')}(d_{-j_{s}}\cdots d_{-j_{1}})u_{0, l', 0, c'}$).  Denote
$$J=J_{\tilde{\mathfrak{g}}}^{\mathfrak{L(g)}}(\lambda, k)\otimes M_{\mathcal{V}}^{\mathfrak{L(g)}}(l', c')+ M_{\tilde{\mathfrak{g}}}^{\mathfrak{L(g)}}(\lambda, k)\otimes J_{\mathcal{V}}^{\mathfrak{L(g)}}(l', c')$$
a submodule of $M_{\tilde{\mathfrak{g}}}^{\mathfrak{L(g)}}(\lambda, k)\otimes M_{\mathcal{V}}^{\mathfrak{L(g)}}(l', c')$. It follows from direct calculations that $J$ is generated by the following singular vectors
\begin{eqnarray*}
&f_{i}^{\langle\lambda+k\Lambda_{i},\check{\alpha_{0}}\rangle+1}u_{\lambda, -\frac{c_{\lambda}}{2(k+g)}, k, \frac{k({\rm dim \,}\mathfrak{g})}{k+g}}\otimes u_{0, l', 0, c'}, \,\, i=0,\ldots,r, \\
&u_{\lambda, -\frac{c_{\lambda}}{2(k+g)}, k, \frac{k({\rm dim \,}\mathfrak{g})}{k+g}}\otimes F_{1}^{(l', c')}(d_{-i_{t}}\cdots d_{-i_{1}})u_{0,l', 0,c'},\\
&u_{\lambda, -\frac{c_{\lambda}}{2(k+g)}, k, \frac{k({\rm dim \,}\mathfrak{g})}{k+g}}\otimes F_{2}^{(l', c')}(d_{-j_{s}}\cdots d_{-j_{1}})u_{0,l', 0,c'}.
\end{eqnarray*}
In view of  the $\mathfrak{L(g)}$-module isomorphism $\pi_1$ in Lemma \ref{coro-2}, $\pi_1^{-1}(J)$ is generated by singular vectors $f_{i}^{\langle\lambda+k\Lambda_{i},\check{\alpha_{0}}\rangle+1}u$, $F_{1}^{(l', c')}(D_{-i_{t}}\cdots D_{-i_{1}})u$, $F_{2}^{(l', c')}(D_{-j_{s}}\cdots D_{-j_{1}})u$ for $i=0,\ldots, r$. What remains is to show that $J$ is the maximal submodule of $ M_{\tilde{\mathfrak{g}}}^{\mathfrak{L(g)}}(\lambda, k)\otimes M_{\mathcal{V}}^{\mathfrak{L(g)}}(l', c')$. From
\begin{align*}
&M_{\tilde{\mathfrak{g}}}^{\mathfrak{L(g)}}(\lambda, k)/J_{\tilde{\mathfrak{g}}}^{\mathfrak{L(g)}}(\lambda, k)\cong L_{\tilde{\mathfrak{g}}}^{\mathfrak{L(g)}}(\lambda, k),\\
&M_{\mathcal{V}}^{\mathfrak{L(g)}}(l', c')/J_{\mathcal{V}}^{\mathfrak{L(g)}}(l', c')\cong L_{\mathcal{V}}^{\mathfrak{L(g)}}(l', c')
\end{align*} and  the $\mathfrak{L(g)}$-module isomorphism $\pi_2$ in Lemma \ref{coro-2}, we obtain
\begin{align*}
\big(M_{\tilde{\mathfrak{g}}}^{\mathfrak{L(g)}}(\lambda, k)\otimes M_{\mathcal{V}}^{\mathfrak{L(g)}}(l', c')\big)/J
=L_{\tilde{\mathfrak{g}}}^{\mathfrak{L(g)}}(\lambda, k)\otimes L_{\mathcal{V}}^{\mathfrak{L(g)}}(l', c')\cong L(\lambda, l, k, c).
\end{align*}
This implies that $J(\lambda, l, k, c)=\pi_1^{-1}(J)$. We complete the proof.
%Take any non-zero element $w=\sum_{i=1}^{s}u_i\otimes v_i\in J'$ with $u_s,v_s\neq\bar{0}$. We proceed by induction on $s$. Consider first the situation when $s=1$. Applying $U(\mathfrak{g})$ to $w$ , we get $u_1\otimes L_{\mathcal{V}}^{\mathfrak{L(g)}}(l', c')\subseteq J'$, which together with $U(\mathcal{V})(u_1\otimes L_{\mathcal{V}}^{\mathfrak{L(g)}}(l', c'))\subseteq J'$ implies that $U(\mathfrak{L(g)})w=J'$. So, $J'$ is irreducible. Assume that $U(\mathfrak{L(g)})w=J'$ for all $w$ with  number items is less than $s$.
\end{proof}

\begin{rema}\label{rem for Ann}
As a direct consequence of Lemma \ref{coro-3}, we know that  ${\rm Ann\,}_{U(\mathfrak{L(g)}_{-})}(\bar{u})$ is the left ideal of $U(\mathfrak{L(g)}_{-})$ generated by $f_i^{<\lambda+k\Lambda_0, \check{\alpha_i}>+1}, E_j$ for $i=0,\ldots, r, j=1,2,$ where
$$F_1^{(l+\frac{c_{\lambda}}{2(k+g)}, c-\frac{k({\rm dim \,}\mathfrak{g})}{k+g})}(D_{-i_{t}}\cdots D_{-i_{1}})u=E_1u,\quad F_2^{(l+\frac{c_{\lambda}}{2(k+g)}, c-\frac{k({\rm dim \,}\mathfrak{g})}{k+g})}(D_{-j_{s}}\cdots D_{-j_{1}})u=E_2u.$$
\end{rema}

In order to illustrate Theorem \ref{theorem2}, we give the following two examples in which $\mathfrak{g}$ is the simple Lie algebra $\mathfrak{sl}_2={\rm span\,}_{\C}\{e,f,h\}$ and the corresponding Lie algebra $\mathfrak{L(sl_2)}$ is just the affine-Virasoro algebra of type $A_1$. Let $\epsilon$ be the fundamental weight of $\mathfrak{sl}_2$. Thus, $g=2$ and $c_{i\epsilon}=\frac{1}{2}(i^2+2i)$ for $i\in\Z_+$.
\begin{exam}\label{exam1}
Let $b-a\not\in\Z$ and $(\lambda, l,k, c, \mu)=(0\epsilon, 0, 1, 2, m\epsilon)$ with $m\in\Z_+$. In this case, we have  $L(\lambda, l, k, c)_{0}=L(0\epsilon)$ is a one-dimensional trivial module and $(l+\frac{c_{\lambda}}{2(k+g)}, c-\frac{3k}{k+g})=(0, 1)$. From \cite[Theorem A]{A}, we see that $J_{\mathcal{V}}(0, 1)=U(\mathcal{V}_{-})d_{-1}u_{0,1}$. Then ${\rm Ann\,}_{U(\mathfrak{L(sl_2)}_{-})}(\bar{u})$ is the left ideal of $U(\mathfrak{L(sl_2)}_{-})$ generated by $D_{-1}, f_0^2, f_1$. We have the following explicit computations
\begin{align*}
0=&D_{-1}\bar{u}=\big(-\frac{1}{3}(-\frac{1}{2}\sum_{i\in\Z}:x(-i)y(i-1)+h(-i)\frac{h}{2}(i-1)+y(-i)x(i-1)):+d_{-1}\big)\bar{u}\nonumber\\
=&(\frac{1}{3}x(-1)y+\frac{1}{6}h(-1)h+\frac{1}{3}y(-1)x+d_{-1})\bar{u}\nonumber\\
=&d_{-1}\bar{u}.
\end{align*}
%Owing to  $b-a\not\in\Z$, it suffices to consider some elements in left ideal of $U(\mathfrak{L(sl_2)}_{-})$ generated by $D_{-1}$ ($d_{-1}$).
For any $n\in\Z$, since
\begin{align*}
&\varphi_n( U(\mathfrak{n}_{-}) d_{-1}\otimes L(m\epsilon))\\
=&\psi_n(d_{-1})\C[f]\bar{u}\otimes  L(m\epsilon)\\
=&(b-a-n-1)L(0\epsilon)\otimes L(m\epsilon)\\
=&L(0\epsilon)\otimes L(m\epsilon),
\end{align*}
the tensor product $\mathfrak{L(sl_2)}$-module $L_{0\epsilon, 0, 1, 2}^{m\epsilon, a,b}$ is irreducible by Theorem \ref{theorem2}.
\end{exam}
%\begin{exam}\label{exam2}
%Let $(\lambda, l,k, c, \mu)= (\epsilon, -1, 1, \frac{3}{4}, m\epsilon)$  with $m\in\Z_+$. In this case, we have  $L(\lambda, l, k, c)_{0}=L(\epsilon)$ and $(l+\frac{k({\rm dim \,}\mathfrak{g})}{k+g}, c+\frac{c_{\lambda}}{2(k+g)})=(0, 1)$. From \cite[Theorem A]{A}, we see that $J_{\mathcal{V}}(0, 1)=U(\mathcal{V}_{-})d_{-1}u_{0,1}$. Then ${\rm Ann\,}_{U(\mathfrak{L(sl_2)}_{-})}(\bar{u})$ is the left ideal of $U(\mathfrak{L(sl_2)}_{-})$ generated by $D_{-1}, f_0, f_1^2$. Explicit computations show that
%\begin{align}\label{cvb}
%D_{-1}\bar{u}=&\big(-\frac{1}{3}(\frac{1}{2}\sum_{i\in\Z}x(-i)y(i-1)+h(-i)\frac{h}{2}(j-1)+y(-i)x(i-1))+d_{-1}\big)\bar{u}\nonumber\\
%=&(-\frac{1}{3}x(-1)y-\frac{1}{6}h(-1)h-\frac{1}{3}y(-1)x+d_{-1})\bar{u}\nonumber\\
%=&(-\frac{1}{3}x(-1)y-\frac{1}{6}h(-1)+d_{-1})\bar{u}=0.
%\end{align}
%Also, we just consider the left ideal of $U(\mathfrak{L(sl_2)}_{-})$ generated by $D_{-1}$. For any $n\in\Z$,
%\begin{align*}
%&\varphi_n( U(\mathfrak{g}\otimes t^{-1}\C[t^{-1}]) U(\mathcal{V}_{-}) U(\mathfrak{n}_{-})(-\frac{1}{3}x(-1)y-\frac{1}{6}h(-1)+d_{-1})\otimes L(m\epsilon))\\
%=&\psi_n(U(\mathcal{V}_{-}))\C[f]\psi_n(d_{-1})\bar{u}\otimes  U(\mathfrak{g})\circ L(m\epsilon)\\
%=&(b-a-n-1)\psi_n(U(\mathcal{V}_{-}))L(0)\otimes L(m\epsilon)\\
%=&L(0)\otimes L(m\epsilon)
%\end{align*}
%\end{exam}

\begin{exam}\label{exam3}
Let $(\lambda, l,k, c, \mu)=(2\epsilon, \frac{3}{2}, 2, \frac{5}{2}, 3\epsilon)$ and $a,b\in\C$. In this case,
we have  $L(\lambda, l, k, c)_{0}=L(2\epsilon)$ and $(l+\frac{c_{\lambda}}{2(k+g)}, c-\frac{3k}{k+g})=(2, 1)$. From Kac determinant formula, we know that $M_{\mathcal{V}}(2, 1)$ is irreducible. So, ${\rm Ann\,}_{U(\mathfrak{L(sl_2)}_{-})}(\bar{u})$ is the left ideal of $U(\mathfrak{L(sl_2)}_{-})$ generated by $f_0, f_1^3$. For any $n\in\Z$, we have
\begin{align*}
&\varphi_n(U(\mathfrak{n}_{-}) f_0+U(\mathfrak{n}_{-})f_1^3)\otimes L(3\epsilon))\\
=&\varphi_n( U(\mathfrak{n}_{-}) f_0\otimes L(3\epsilon))\\
=&\varphi_n(\C[f]f_0\otimes  L(3\epsilon)).
\end{align*}
Using this and  the same arguments as in \cite[Theorems 7.8, 7.9, Corollary 7.10]{GZ} (in more general situations.), we can prove that
\begin{align*}
&L_{2\epsilon, \frac{3}{2}, 2, \frac{5}{2}}^{3\epsilon, a,b} \,\,{\rm is \,\, irreducible}\\
\Longleftrightarrow &\varphi_n(\C[f]f_0\otimes L(3\epsilon))=L(2\epsilon)\otimes L(3\epsilon)\\
\Longleftrightarrow &\C[f](\bar{u}\otimes eL(3\epsilon))=L(2\epsilon)\otimes L(3\epsilon)
\end{align*}
and $\C[f](\bar{u}\otimes eL(3\epsilon))$ contains all highest weight vectors of $L(2\epsilon)\otimes L(3\epsilon)$. Therefore, the tensor product $\mathfrak{L(sl_2)}$-module $L_{2\epsilon, \frac{3}{2}, 2, \frac{5}{2}}^{3\epsilon, a,b}$ is irreducible.
\end{exam}
\begin{rema}
As all weight spaces of  $L(\lambda, l, k, c)\otimes L_{a,b}(\mu)$ are infinite-dimensional and each weight vector of the positive parts of $\mathfrak{L(g)}$ acts non-locally finitely on the  tensor product modules, the tensor product modules are different from the other known irreducible weight modules (cf. \cite{K, LPX, T}). Hence, these tensor product modules are new irreducible weight modules.
\end{rema}

%\begin{theo}\label{theoo1121}
%$L_{\lambda, l, k, c}^{\mu, a,b}$  is irreducible if and only if $\varphi_n( U(\mathfrak{n}_{-})f_0^{k-<\lambda,\check{\theta}>+1}\otimes L(\mu))=L(\lambda)\otimes L(\mu)$ for all  $n\in\Z$.
%\end{theo}

\section{Isomorphism classes of the tensor product modules}
Let $\lambda,\lambda'\in \mathfrak{h}^*,\mu,\mu'\in\Lambda^{+}\setminus\{0\}$ and $k,c,l,a,b,k',c',l',a',b'\in\C$ with $0\leq{\rm Re\,}a,{\rm Re\,}a'<1$. In this section, we will determine the necessary and sufficient conditions for two tensor product modules $L(\lambda, l, k, c)\otimes L_{a,b}(\mu)$ and  $L(\lambda', l', k', c')\otimes L_{a',b'}(\mu')$ to be isomorphic. As before, we identify  $L(\lambda, l, k, c)\otimes L_{a,b}(\mu)$ and $L(\lambda', l', k', c')\otimes L_{a',b'}(\mu')$   with  $L_{\lambda, l, k, c}^{\mu, a,b}$ and $L_{\lambda', l', k', c'}^{\mu', a',b'}$    , respectively, via \eqref{equ21}.
\begin{theo}\label{theoo21}
Let $\lambda,\lambda'\in \mathfrak{h}^*,\mu,\mu'\in\Lambda^{+}\setminus\{0\}$ and $k,c,l,a,b,k',c',l',a',b'\in\C$ with $0\leq{\rm Re\,}a,{\rm Re\,}a'<1$. Then $L_{\lambda, l, k, c}^{\mu, a,b}\cong L_{\lambda', l', k', c'}^{\mu', a',b'}$ if and only if $(\lambda, l, k, c, \mu, a, b)=(\lambda', l', k', c', \mu', a', b')$.
\end{theo}
\begin{proof}
The sufficiency is obvious and it suffices to show the necessity.  Let $\bar{u}\,({\rm resp.\,} \bar{u}')$ and $v_{\mu}\,({\rm resp.\,}v_{\mu'})$ be the highest weight vectors of
$$L(\lambda, l, k, c)\,({\rm resp.\,} L(\lambda', l', k', c'))\quad {\rm and\,}\quad L(\mu)\,({\rm resp.\,} L(\mu')).$$  Assume that $\phi:L_{\lambda, l, k, c}^{\mu, a,b}\rightarrow L_{\lambda', l', k', c'}^{\mu', a',b'}$ is an $\mathfrak{L(g)}$-module  isomorphism. It is easy to see that $k=k',c=c'$. Fix any $p\in\Z$. As $\phi(\bar{u}\otimes v_{\mu}\otimes t^p)$ and $\bar{u}\otimes v_{\mu}\otimes t^p$ are of the same weight with respect to the action of $d_0$, we can assume that $\phi(\bar{u}\otimes v_{\mu}\otimes t^p)=\sum_{i=0}^sw_i\otimes v_i\otimes t^{p'}$ with
\begin{equation}\label{vnma}
a+l+p=a'+l'+p',
\end{equation}
$0\neq v_i\in L(\mu')$ and $w_0,\ldots,w_s\in  L(\lambda', l', k', c')$ such that $w_0,\ldots,w_s$ are linearly independent.
There exists $m_0\in\Z_+$ such that $\mathfrak{n}_{+}(m)(w_i)=0$ for $m\geq m_0$ and $i=0,\ldots,s$. From
$$0=\phi(\mathfrak{n}_{+}(m)(\bar{u}\otimes v_{\mu}\otimes t^p))=\mathfrak{n}_{+}(m)\phi(\bar{u}\otimes v_{\mu}\otimes t^p)=\sum_{i=0}^sw_i\otimes \mathfrak{n}_{+}v_i\otimes t^{p'+m},\ \ \ \ \forall m\geq m_0$$
and the linear independence of $w_0,\ldots,w_s$, we know that $\mathfrak{n}_{+}v_i=0$ for $i=0,\ldots,s$. This allows us to assume that
$\phi(\bar{u}\otimes v_{\mu}\otimes t^p)=\sum_{i=0}^rR_i\bar{u}'\otimes v_{\mu'}\otimes t^{p'}$ with ${\rm deg\,}R_i\in U(\mathfrak{L(g)}_{-})_{-i}$ for $0\leq i\leq r$ and $R_l\bar{u}'\neq 0$. For any $m,n\geq r+1,x\in\mathfrak{n}_{-}$, we have
\begin{align*}
&\phi(\bar{u}\otimes xv_{\mu}\otimes t^{p+m+n})\\
=&\phi(x(m+n)(\bar{u}\otimes v_{\mu}\otimes t^p))\\
=&x(m+n)\phi(\bar{u}\otimes v_{\mu}\otimes t^p)\\
=&\sum_{i=0}^rR_i\bar{u}'\otimes xv_{\mu'}\otimes t^{p'+m+n}
\end{align*}
and
\begin{align*}
&(a+p+n+bm)\phi(\bar{u}\otimes xv_{\mu}\otimes t^{p+m+n})\\
=&\phi(d_mx(n)(\bar{u}\otimes v_{\mu}\otimes t^p))\\
=&d_mx(n)\phi(\bar{u}\otimes v_{\mu}\otimes t^p)\\
=&(a'+p'+n+i+b'm)\sum_{i=0}^rR_i\bar{u}'\otimes xv_{\mu'}\otimes t^{p'+m+n}.
\end{align*}
Comparing the above two equations, we obtain that
$$a+p+n+bm=a'+p'+n+r+b'm,\ \ \ \ \forall m,n\geq r+1.$$
Since this equation holds for all $m\geq r+1$, we have
$$a=a', b=b', p=p'+r.$$
The last equality means  $p\geq p'$. By considering  $\phi^{-1}$ in the above arguments, we deduce that $p'\geq p$, implying $p=p'$ and then $r=0$. It follows immediately
from \eqref{vnma} that $l=l'$. Therefore, for any $p\in\Z$, we get $\phi(\bar{u}\otimes v_{\mu}\otimes t^p)=R_0\bar{u}'\otimes v_{\mu'}\otimes t^{p}$ and we will replace $R_0$ with $R_{0,p}$ because it also depends on $p$. Recall from the proof of Lemma  \ref{lemma1} that
$$\sum_{p\in\Z}U(\mathfrak{L(g)}_{-})U(\mathfrak{L(g)}_{+})(\bar{u}\otimes v_{\mu}\otimes t^p)=L_{\lambda, l, k, c}^{\mu, a, b}.$$
This along with the fact that $\phi$ is an isomorphism gives that
$$\sum_{p\in\Z}U(\mathfrak{L(g)}_{-})U(\mathfrak{L(g)}_{+})(R_{0,p}\bar{u}'\otimes v_{\mu'}\otimes t^p)=L_{\lambda', l', k', c'}^{\mu', a', b'},$$
from which we know that $R_{0,p}\bar{u}'$ is a nonzero scalar multiple of $\bar{u}'$, that is, there exist $r_p\in\C^*$ such that $\phi(\bar{u}\otimes v_{\mu}\otimes t^p)=r_p\bar{u}'\otimes v_{\mu'}\otimes t^{p}$ for all $p\in\Z$. For any $n\in\N, h\in \mathfrak{h}$, we have
\begin{align*}
&\mu(h)r_{p+n}\bar{u}'\otimes v_{\mu'}\otimes t^{p+n}=\phi(h(n)(\bar{u}\otimes v_{\mu}\otimes t^p))\\
&\ \ \ \ =h(n)\phi(\bar{u}\otimes v_{\mu}\otimes t^p)=\mu'(h)r_p\bar{u}'\otimes v_{\mu'}\otimes t^{p+n}.
\end{align*}
This yields $\mu(h)r_{p+n}=\mu'(h)r_p$, from which one can deduce that $\mu=\mu'$. For any $h\in\mathfrak{h}$, from
\begin{align*}
&(\lambda(h)+\mu(h))r_p\bar{u}'\otimes v_{\mu'}\otimes t^{p}=\phi(h(\bar{u}\otimes v_{\mu}\otimes t^p))\\
&\ \ \ \ =h\phi(\bar{u}\otimes v_{\mu}\otimes t^p)=(\lambda'(h)+\mu(h))r_p\bar{u}'\otimes v_{\mu'}\otimes t^{p},
\end{align*}
we see that $\lambda=\lambda'$, completing the proof.

\end{proof}

\subsection*{Acknowledgements}
Qiufan Chen would like to thank  professor Haibo Chen  for raising the problem in the discussion.

\end{document}